\documentclass[12pt,reqno]{amsart}

\usepackage{tikz}
\usepackage{xcolor}

\usepackage{latexsym,amsmath,amssymb,epic}
\usepackage{latexsym,amssymb,epic}
\usepackage{amsmath,amsthm}
\usepackage{color}
\usepackage{nicematrix}

\usepackage[left=3.75cm,right=3.75cm]{geometry}
\usepackage{etoolbox}
\usepackage{enumitem}
\usepackage[noadjust]{cite}
\usepackage{hyperref}

\allowdisplaybreaks[4]

\numberwithin{equation}{section}

\theoremstyle{theorem}
\newtheorem{theorem}{Theorem}
\newtheorem*{theorem*}{Theorem}

\newtheorem{lemma}[theorem]{Lemma}

\theoremstyle{definition}

\newtheorem*{example*}{Example}

\newtheorem*{conjecture*}{Conjecture}
\newtheorem{remark}{Remark}
\newtheorem*{remark*}{Remark}
\newtheorem*{remarks*}{Remarks}

\makeatletter
\patchcmd{\section}{\scshape}{\bfseries\boldmath}{}{}
\patchcmd{\subsection}{\bfseries}{\bfseries\boldmath}{}{}
\renewcommand{\@secnumfont}{\bfseries}
\makeatother

\begin{document}
	
\title[Congruences for POND and PEND Partitions]{Elementary Proofs of Congruences for POND and PEND Partitions}

\author[J. A. Sellers]{James A. Sellers}
\address[J. A. Sellers]{Department of Mathematics and Statistics, University of Minnesota Duluth, Duluth, MN 55812, USA}
\email{jsellers@d.umn.edu}
	
	
	
\maketitle
\begin{abstract}
Recently, Ballantine and Welch considered various generalizations and refinements of POD and PED partitions.  These are integer partitions wherein the odd parts must be distinct (in the case of POD partitions) or the even parts must be distinct (in the case of PED partitions).  In the process, they were led to consider two classes of integer partitions which are, in some sense, the ``opposite'' of POD and PED partitions.  They labeled these POND and PEND partitions, which are integer partitions wherein the odd parts cannot be distinct (in the case of POND partitions) or the even parts cannot be distinct (in the case of PEND partitions).  In this work, we study these two types of partitions from an arithmetic perspective.  Along the way, we are led to prove the following two infinite families of Ramanujan--like congruences: 
For all  $\alpha \geq 1$ and all $n\geq 0,$  
\begin{align*}
pond\left(3^{2\alpha +1}n+\frac{23\cdot 3^{2\alpha}+1}{8}\right) 
&\equiv 0 \pmod{3}, \textrm{\ \ \ and} \\
pend\left(3^{2\alpha +1}n+\frac{17\cdot 3^{2\alpha}-1}{8}\right) 
&\equiv 0 \pmod{3}
\end{align*}
where $pond(n)$ counts the number of POND partitions of weight $n$ and $pend(n)$ counts the number of PEND partitions of weight $n$.

All of the proof techniques used herein are elementary, relying on classical $q$-series identities and generating function manipulations, along with mathematical induction.
\end{abstract}

\section{Introduction}  
In the study of integer partitions, the partitions wherein the parts are distinct have long played a key role, due in large part to Euler's famous identity which states that the number of partitions of weight $n$ into distinct parts equals the number of partitions of weight $n$ into odd parts.  One of the most obvious refinements in this regard is to require distinct parts based on parity; i.e., to require either all of the even parts to be distinct or all of the odd parts to be distinct (while allowing the frequency of the other parts to be unrestricted).  This leads to two types of partitions, those that we will call PED partitions (wherein the even parts must be distinct and the odd parts are unrestricted) and POD partitions (wherein the odd parts must be distinct and the even parts are unrestricted).  We then define two corresponding enumerating functions, $ped(n)$ which counts the number of PED partitions of weight $n$, and $pod(n)$ which counts the number of POD partitions of weight $n.$ These two functions have been studied from a variety of perspectives; the interested reader may wish to see \cite{And2009, AHS, BalMer, BalWelDM, Chen1, Chen2, CuiGu, CuiGuMa, FangXueYao, HS2010, HS_2014_JIS, KeithZan, Merca, RaduSel, Xia, Wang} for examples of work on identities involving, and arithmetic properties satisfied by, $ped(n)$ and $pod(n)$.  

Recently, Ballantine and Welch \cite{BalWel} generalized and refined these two functions in numerous ways.  One of the outcomes of their work was to consider integer partitions which are, in some sense, the ``opposite'' of PED partitions and POD partitions.  Namely, they considered PEND partitions and POND partitions, wherein the even (respectively, odd) parts are {\bf not allowed} to be distinct.  In a vein similar to that shared above, we let $pend(n)$ denote the number of PEND partitions of weight $n$, and $pond(n)$ denote the number of POND partitions of weight $n.$  The first several values of $pend(n)$ appear in the OEIS \cite[A265254]{OEIS}, while the first several values of $pond(n)$ appear in \cite[A265256]{OEIS}.

It is worthwhile to share additional historical thoughts to place PEND and POND partitions in context.  In his classic \textit{Combinatory Analysis} \cite{MacM}, P. A. MacMahon proved that, for all $n\geq 0$, the number of partitions of weight $n$ wherein no part appears with multiplicity one equals the number of partitions of weight $n$ where all parts must be even or congruent to 3 modulo 6.  As an aside, we note that numerous mathematicians have since generalized this theorem of MacMahon and have provided proofs of these results using both generating functions (which was MacMahon's original approach) as well as combinatorial arguments.  The first half of the statement of MacMahon's theorem involves the function which counts the number of partitions wherein no part appears with multiplicity one, i.e., no part is allowed to be distinct.  It is in this sense that POND and PEND partitions provide a natural, parity--based refinement of the partitions considered by MacMahon.

At the end of their paper, Ballantine and Welch \cite{BalWel} shared the following possibilities for future work: 
\begin{quotation}
In particular, we note two areas of interest. The first is examining the arithmetic properties of these generalizations. Much work has been done in studying arithmetic properties of PED and POD partitions... Hence, this would be a natural topic of further study...    
\end{quotation}
In light of this suggestion from Ballantine and Welch, our overarching goal in this work is to study $pond(n)$ and $pend(n)$ from an arithmetic perspective.  With this in mind, we will first prove the following Ramanujan--like congruences satisfied by $pond(n)$ and $pend(n)$:

\begin{theorem}
\label{thm:pond_congs}
For all $n\geq 0,$
\begin{eqnarray}
pond(3n+2) &\equiv & 0 \pmod{2}, \label{pond_mod2} \\
pond(27n+26) &\equiv & 0 \pmod{3}, \textrm{\ \ and} \label{pond_mod3} \\
pond(3n+1) &\equiv & 0 \pmod{4} \label{pond_mod4}.
\end{eqnarray}
\end{theorem}

\begin{theorem}
\label{thm:pend_cong}
For all $n\geq 0,$    
\begin{equation*}
pend(27n+19) \equiv  0 \pmod{3}.    
\end{equation*}
\end{theorem}

We will then prove that each of these two functions satisfies an internal congruence modulo 3.

\begin{theorem}
\label{thm:pond_internal}
For all $n\geq 0,$ $pond(27n+17) \equiv pond(3n+2) \pmod{3}.$     
\end{theorem}

\begin{theorem}
\label{thm:pend_internal}
For all $n\geq 0,$ $pend(27n+10) \equiv pend(3n+1) \pmod{3}.$     
\end{theorem}

Finally, with the above results in hand, we will prove the following infinite families of non--nested Ramanujan--like congruences modulo 3 by induction.  

\begin{theorem}
\label{thm:pond_infinite_family_congs}
For all  $\alpha \geq 1$ and all $n\geq 0,$   
$$
pond\left(3^{2\alpha +1}n+\frac{23\cdot 3^{2\alpha}+1}{8}\right) \equiv 0 \pmod{3}.
$$
\end{theorem}

\begin{theorem}
\label{thm:pend_infinite_family_congs}
For all  $\alpha \geq 1$ and all $n\geq 0,$   
$$
pend\left(3^{2\alpha +1}n+\frac{17\cdot 3^{2\alpha}-1}{8}\right) \equiv 0 \pmod{3}.
$$
\end{theorem}

Section \ref{sec:prelims} is devoted to providing the tools necessary for the remainder of the paper.  In Section \ref{sec:pond}, we prove Theorems \ref{thm:pond_congs}, \ref{thm:pond_internal}, and \ref{thm:pond_infinite_family_congs}.  In Section \ref{sec:pend}, we prove Theorems \ref{thm:pend_cong}, \ref{thm:pend_internal}, and \ref{thm:pend_infinite_family_congs}.
All of the proof techniques used herein are elementary, relying on classical $q$-series identities and generating function manipulations, along with mathematical induction.


\section{Preliminaries}
\label{sec:prelims}

Throughout this work, we will use the following shorthand notation for $q$-Pochhammer symbols:
\begin{align*}
    f_r := (q^r;q^r)_\infty = (1-q^r)(1-q^{2r})(1-q^{3r})\dots
\end{align*}

In order to prove the congruences mentioned above, several important 3--dissections of various $q$--series will be needed.  These results will allow us to write the necessary generating functions in an appropriate fashion. We now catalog these results here.  

\begin{lemma}
\label{f2_over_f1f4}   
We have 
$$
\frac{f_2}{f_1f_4} = \frac{f_{18}^9}{f_3^2f_9^3f_{12}^2f_{36}^3} + q\frac{f_6^2f_{18}^3}{f_3^3f_{12}^3} + q^2\frac{f_6^4f_9^3f_{36}^3}{f_3^4f_{12}^4f_{18}^3}.
$$
\end{lemma}
\begin{proof}
A proof of this identity appears in \cite[Lemma 2.1]{Toh}.    
\end{proof}
\begin{lemma}
\label{f1f2} 
We have 
$$
f_1f_2 = \frac{f_6f_9^4}{f_3f_{18}^2} - qf_{9}f_{18} -2q^2\frac{f_3f_{18}^4}{f_6f_9^2}.
$$
\end{lemma}
\begin{proof}
A proof of this identity can be found in \cite{HS_2014_JIS}.
\end{proof}
\begin{lemma}
\label{1_over_f1f2} 
We have 
$$
\frac{1}{f_1f_2} = \frac{f_9^9}{f_3^6f_6^2f_{18}^3} + q\frac{f_9^6}{f_3^5f_6^3} + 3q^2\frac{f_9^3f_{18}^3}{f_3^4f_6^4} -2q^3\frac{f_{18}^6}{f_3^3f_6^5} +4q^4\frac{f_{18}^9}{f_3^2f_6^6f_9^3}.
$$
\end{lemma}
\begin{proof}
This lemma is equivalent to \cite[Equation (39)]{MG}.  
\end{proof}
\begin{lemma}
\label{f2squared_over_f1} 
We have 
$$
\frac{f_2^2}{f_1} = \frac{f_6f_9^2}{f_3f_{18}} + q\frac{f_{18}^2}{f_9}.
$$
\end{lemma}
\begin{proof}
For a proof of this result, see \cite[(14.3.3)]{H}.      
\end{proof}
\begin{lemma}
\label{f2_over_f1squared} 
We have 
$$
\frac{f_2}{f_1^2} = \frac{f_6^4f_9^6}{f_3^8f_{18}^3} + 2q\frac{f_6^3f_9^3}{f_3^7}+ 4q^2\frac{f_6^2f_{18}^3}{f_3^6}.
$$
\end{lemma}
\begin{remark}
Note that 
$$
\frac{f_2}{f_1^2} = \sum_{n=0}^\infty \overline{p}(n)q^n
$$
where $\overline{p}(n)$ is the number of overpartitions of $n.$
\end{remark}
\begin{proof}
For a proof of Lemma \ref{f2_over_f1squared}, see \cite[Theorem 1]{HS_2005_JCMCC}.    
\end{proof}
\begin{lemma}
\label{f4_over_f1}
We have    
$$
\frac{f_4}{f_1} = \frac{f_{12}f_{18}^4}{f_3^3f_{36}^2} +q\frac{f_6^2f_9^3f_{36}}{f_3^4f_{18}^2}+2q^2\frac{f_6f_{18}f_{36}}{f_3^3}.
$$
\end{lemma}
\begin{remark}
Note that 
$$
\frac{f_4}{f_1} = \sum_{n=0}^\infty ped(n)q^n
$$
where $ped(n)$ is the number of partitions of $n$ wherein even parts are distinct (as mentioned in the introductory comments above). 
\end{remark}
\begin{proof}
Lemma \ref{f4_over_f1} follows from \cite[Theorem 3.1]{AHS} and \cite[(33.2.6)]{H}.   
\end{proof}
One additional $q$--series identity will be beneficial in the proof of Theorem \ref{thm:pond_internal}.
\begin{lemma}
\label{H_22.7.5}
We have 
$$
\frac{f_3^3}{f_1}-q\frac{f_{12}^3}{f_4} = \frac{f_4^3f_6^2}{f_2^2f_{12}}.
$$
\end{lemma}
\begin{proof}
This identity appears in \cite[(22.7.5)]{H}.      
\end{proof}
Lastly, we will utilize the following result which, at its core, relies on the binomial theorem and the divisibility properties of various binomial coefficients. 
\begin{lemma}
\label{lemma:binthm_divisibility}
For all primes $p$ and all $j,k,m\geq 1$, 
$f_{m}^{p^j k} \equiv f_{pm}^{p^{j-1}k} \pmod{p^j}$.
\end{lemma}
With all of these tools in hand, we are now in a position to prove the theorems listed above. 


\section{Congruences for $pond(n)$}
\label{sec:pond}

We begin by considering the function $pond(n).$  Although one can derive the generating function for $pond(n)$ from the work of Ballantine and Welch \cite{BalWel}, we provide a proof of the result here for the sake of completeness.  

\begin{theorem}
\label{thm:pond_genfn}
We have
$$
\sum_{n=0}^{\infty} pond(n)q^n =\frac{f_4f_6^2}{f_2^2f_3f_{12}}.
$$
\end{theorem}

\begin{proof}
By definition, 
\begin{align*}
\sum_{n=0}^\infty pond(n)q^n 
&= 
\frac{1}{f_2}\prod_{i=1}^\infty \left( \frac{1}{1-q^{2i-1}} - q^{2i-1} \right) \\
&= 
\frac{1}{f_2}\prod_{i=1}^\infty \left( \frac{1-q^{2i-1}+q^{4i-2}}{1-q^{2i-1}}  \right) \\
&= 
\frac{1}{f_2}\prod_{i=1}^\infty \left( \frac{1+q^{6i-3}}{(1+q^{2i-1})(1-q^{2i-1})}  \right) \\
&= 
\frac{1}{f_2}\cdot \frac{(-q^3;q^6)_\infty}{(q^2;q^4)_\infty} \\
&= 
\frac{1}{f_2}\cdot \frac{f_4}{f_2}\cdot \frac{(q^6;q^{12})_\infty}{(q^3;q^{6})_\infty} \\
&= 
\frac{f_4}{f_2^2}\cdot \frac{f_6}{f_{12}}\cdot \frac{f_6}{f_3} \\
&= 
\frac{f_4f_6^2}{f_2^2f_3f_{12}}.
\end{align*}
\end{proof}
We can now move to a proof of Theorem \ref{thm:pond_congs}.  

\begin{proof}(of Theorem \ref{thm:pond_congs})
Our first goal is to 3--dissect the generating function for $pond(n)$.  Note that 
\begin{align*}
 \sum_{n=0}^{\infty} pond(n)q^n 
 &=
 \frac{f_4f_6^2}{f_2^2f_3f_{12}}   \\
 &= 
 \frac{f_4}{f_2^2}\cdot \frac{f_6^2}{f_3f_{12}}   \\
 &= 
 \left( \frac{f_{12}^4f_{18}^6}{f_6^8f_{36}^3} + 2q^2\frac{f_{12}^3f_{18}^3}{f_6^7}+ 4q^4\frac{f_{12}^2f_{36}^3}{f_6^6} \right)\cdot \frac{f_6^2}{f_3f_{12}} 
\end{align*}
thanks to Lemma \ref{f2_over_f1squared}.  This means we know the following:  
\begin{align*}
\sum_{n=0}^\infty pond(3n)q^{3n} 
&= 
\frac{f_6^2}{f_3f_{12}}\cdot \frac{f_{12}^4f_{18}^6}{f_6^8f_{36}^3}, \\
\sum_{n=0}^\infty pond(3n+1)q^{3n+1} 
&= 
\frac{f_6^2}{f_3f_{12}}\cdot 4q^4\frac{f_{12}^2f_{36}^3}{f_6^6}, \textrm{\ \ \ and} \\
\sum_{n=0}^\infty pond(3n+2)q^{3n+2} 
&= 
\frac{f_6^2}{f_3f_{12}}\cdot 2q^2\frac{f_{12}^3f_{18}^3}{f_6^7}.    
\end{align*}
This is equivalent to the following 3--dissection for the generating function for $pond(n)$:
\begin{align}
\sum_{n=0}^\infty pond(3n)q^{n} 
&= 
\frac{f_2^2}{f_1f_4}\cdot \frac{f_4^4f_6^6}{f_2^8f_{12}^3} = \frac{f_4^3f_6^6}{f_1f_2^6f_{12}^3}, \label{pond3n0_genfn}\\
\sum_{n=0}^\infty pond(3n+1)q^{n} 
&= 
4q\frac{f_2^2}{f_1f_{4}}\cdot \frac{f_{4}^2f_{12}^3}{f_2^6} = 4q\frac{f_4f_{12}^3}{f_1f_2^4}, \textrm{\ \ \ and} \label{pond3n1_genfn} \\
\sum_{n=0}^\infty pond(3n+2)q^{n} 
&= 
2\frac{f_2^2}{f_1f_{4}}\cdot \frac{f_{4}^3f_{6}^3}{f_2^7}  = 2\frac{f_4^2f_6^3}{f_1f_2^5}.   \label{pond3n2_genfn} 
\end{align}
We pause here to note that (\ref{pond3n2_genfn}) implies  (\ref{pond_mod2}) while (\ref{pond3n1_genfn}) implies (\ref{pond_mod4}).  Thus, in order to complete the proof of Theorem \ref{thm:pond_congs}, we simply need to prove (\ref{pond_mod3}), and this requires us to 3--dissect the generating function for $pond(3n+2)$ which appears in (\ref{pond3n2_genfn}):  
\begin{align*}
&
\sum_{n=0}^\infty pond(3n+2)q^{n} \notag\\
&= 
2\frac{f_4^2f_6^3}{f_1f_2^5}  \notag \\
&\equiv 
2\frac{f_4^2f_2^9}{f_1f_2^5} \pmod{3} \textrm{\ \ thanks to Lemma \ref{lemma:binthm_divisibility}}  \notag \\
&= 
2\frac{f_4^2f_2^4}{f_1} \\ 
&= 
2(f_2f_4)^3\cdot \frac{f_2^2}{f_1}\cdot \frac{1}{f_2f_4} \notag \\
&\equiv 
2f_6f_{12}\cdot \frac{f_2^2}{f_1}\cdot \frac{1}{f_2f_4} \pmod{3}  \notag \\
&\equiv
2f_6f_{12}\left( \frac{f_6f_9^2}{f_3f_{18}} + q\frac{f_{18}^2}{f_9} \right)   \\
& \ \ \ \ \ \ \ \ \ \ \times \left( \frac{f_{18}^9}{f_6^6f_{12}^2f_{36}^3} + q^2\frac{f_{18}^6}{f_6^5f_{12}^3} + q^6\frac{f_{36}^6}{f_6^3f_{12}^5} +q^8\frac{f_{36}^9}{f_6^2f_{12}^6f_{18}^3} \right)  \pmod{3} \notag \\
\end{align*}
thanks to Lemmas \ref{1_over_f1f2} and \ref{f2squared_over_f1}.  
Thus, we know 
\begin{align*}
\sum_{n=0}^\infty pond(9n+8)q^{3n+2} 
&\equiv 
2f_6f_{12}\cdot \frac{f_6f_9^2}{f_3f_{18}} \left( q^2\frac{f_{18}^6}{f_6^5f_{12}^3} +q^8\frac{f_{36}^9}{f_6^2f_{12}^6f_{18}^3} \right)  \pmod{3}.
\end{align*}
Therefore, 
\begin{align*}
\sum_{n=0}^\infty pond(9n+8)q^{n} 
&\equiv 
2\frac{f_2^2f_3^2f_4}{f_1f_{6}} \left( \frac{f_{6}^6}{f_2^5f_{4}^3}  +q^2\frac{f_{12}^9}{f_2^2f_{4}^6f_{6}^3} \right)  \pmod{3} \\
&=
2\frac{f_3^2}{f_1f_4^2} \left( \frac{f_{6}^5}{f_2^3}  +q^2\frac{f_{12}^9}{f_{4}^3f_{6}^4} \right)  \\
&\equiv 
2\frac{f_3^2}{f_1f_4^2} \left( \frac{f_{6}^5}{f_6}  +q^2\frac{f_{12}^9}{f_{12}f_{6}^4} \right)  \pmod{3} \textrm{\ \ \ using Lemma \ref{lemma:binthm_divisibility}} \\
&= 
2\frac{f_3^2f_4}{f_1f_4^3} \left( f_{6}^4  +q^2\frac{f_{12}^8}{f_{6}^4} \right) \\
&\equiv 
2\frac{f_4}{f_1}\cdot \frac{f_3^2}{f_{12}} \left( f_{6}^4  +q^2\frac{f_{12}^8}{f_{6}^4} \right)  \pmod{3}.  
\end{align*}
We now use Lemma \ref{f4_over_f1} to see that 
\begin{align*}
\sum_{n=0}^\infty pond(27n+26)q^{3n+2} 
&\equiv 
2\frac{f_3^2}{f_{12}} \left( q^2\frac{f_{12}^8}{f_{6}^4}\cdot \frac{f_{12}f_{18}^4}{f_3^3f_{36}^2} + 2q^2\frac{f_6^5f_{18}f_{36}}{f_3^3} \right)  \pmod{3}
\end{align*}
so that 
\begin{align*}
\sum_{n=0}^\infty pond(27n+26)q^{n} 
&\equiv 
2\frac{f_1^2}{f_{4}} \left( \frac{f_{4}^9f_6^4}{f_{2}^4f_1^3f_{12}^2} + 2\frac{f_2^5f_{6}f_{12}}{f_1^3} \right)  \pmod{3} \\
&= 
\frac{f_2^5}{f_1f_4}\left( 2\frac{f_4^9f_6^4}{f_2^9f_{12}^2} + 4f_6f_{12} \right) \\
&\equiv 
\frac{f_2^5}{f_1f_4}\left( 2\frac{f_{12}^3f_6^4}{f_6^3f_{12}^2} + 4f_6f_{12} \right) \pmod{3}\\
&\equiv 
\frac{f_2^5}{f_1f_4}\left( 6f_6f_{12} \right) \pmod{3}\\
&\equiv 
0 \pmod{3}.  
\end{align*}
This completes the proof of (\ref{pond_mod3}) and, therefore, Theorem \ref{thm:pond_congs}.
\end{proof}
Equation (\ref{pond_mod3})  will serve as the base case for the proof by induction of Theorem \ref{thm:pond_infinite_family_congs}.  However, before we turn to the proof of Theorem \ref{thm:pond_infinite_family_congs}, we first prove Theorem \ref{thm:pond_internal} which will be the ``engine'' for that proof by induction.  
\begin{proof}(of Theorem \ref{thm:pond_internal})
Our goal is to prove that, for all $n\geq 0,$ $$pond(27n+17) \equiv pond(3n+2) \pmod{3}.$$  
From our work above, we know 
\begin{equation}
\label{3n2_mod3}
\sum_{n=0}^\infty pond(3n+2)q^n \equiv 2\frac{f_4^2}{f_1}(f_2^4) \pmod{3}. 
\end{equation}
Next, we need to determine a corresponding congruence for the generating function for $pond(27n+17).$  In our earlier work, we showed that 
$$
\sum_{n=0}^\infty pond(9n+8)q^n 
\equiv 
2\frac{f_4}{f_1}\cdot \frac{f_3^2}{f_{12}} \left( f_{6}^4  +q^2\frac{f_{12}^8}{f_{6}^4} \right)  \pmod{3}.  
$$
We can then use Lemma \ref{f4_over_f1} to see that 
\begin{align*}
& \sum_{n=0}^\infty pond(9(3n+1)+8) q^{3n+1} \\
&\equiv  2 \frac{f_3^2}{f_{12}} \left( f_{6}^4\cdot q\frac{f_6^2f_9^3f_{36}}{f_3^4f_{18}^2}  +q^2\frac{f_{12}^8}{f_{6}^4}\cdot 2q^2\frac{f_6f_{18}f_{36}}{f_3^3} \right)   \pmod{3}
\end{align*}
or 
\begin{align}
\sum_{n=0}^\infty pond(27n+17) q^{n} 
&\equiv  
2 \frac{f_1^2}{f_{4}} \left(  \frac{f_2^6f_3^3f_{12}}{f_1^4f_{6}^2}  +2q\frac{f_4^8f_{6}f_{12}}{f_1^3f_2^3} \right)   \pmod{3} \notag\\
&\equiv  
2 \frac{f_2^6f_3^3f_{12}}{f_1^2f_4f_6^2} + 4q \frac{f_4^7f_6f_{12}}{f_1f_2^3}   \pmod{3} \notag\\
&\equiv  
2 \frac{f_2^6f_1^9f_{4}^3}{f_1^2f_4f_2^6} + 4q \frac{f_4^7f_2^3f_{4}^3}{f_1f_2^3}   \pmod{3} \notag\\
&\equiv  
2 f_1^7f_{4}^2 + 4q \frac{f_4^{10}}{f_1}   \pmod{3} \notag\\
&= 
2\frac{f_4^2}{f_1}\left(f_1^8+2qf_4^8\right). \label{27n17_mod3}
\end{align}
Therefore, in order to prove this theorem, we know from (\ref{3n2_mod3}) and (\ref{27n17_mod3}) that we must show the following:  
$$
2\frac{f_4^2}{f_1}\left(f_1^8+2qf_4^8\right) \equiv 2\frac{f_4^2}{f_1}(f_2^4) \pmod{3}
$$
or 
$$
f_1^8+2qf_4^8 \equiv f_2^4\pmod{3}.
$$
To complete this proof, we are reminded of Lemma \ref{H_22.7.5}:  
$$
\frac{f_3^3}{f_1}-q\frac{f_{12}^3}{f_4} = \frac{f_4^3f_6^2}{f_2^2f_{12}}.
$$
Note that this implies that 
$$
\frac{f_1^9}{f_1}+2q\frac{f_{4}^9}{f_4} \equiv \frac{f_{12}f_2^6}{f_2^2f_{12}} \pmod{3}
$$
or
$$
f_1^8+2qf_{4}^8 \equiv f_2^4 \pmod{3}
$$
which is the desired result.  
\end{proof}
With Theorems \ref{thm:pond_congs} and \ref{thm:pond_internal} in hand, we can now turn to proving the infinite family of Ramanujan--like congruences modulo 3 satisfied by $pond(n)$. 

\begin{proof}(of Theorem \ref{thm:pond_infinite_family_congs}) 
We prove this theorem by induction on $\alpha$.  Note that the base case, $\alpha=1$, which corresponds to the arithmetic progression 
$$
3^3n + \frac{23\cdot 3^{2}+1}{8} = 27n+26,
$$ 
has already been proved in Theorem \ref{thm:pond_congs} above.   
Thus, we assume that, for some $\alpha\geq 1$ and all $n\geq 0$, 
$$
pond\left(3^{2\alpha +1}n+\frac{23\cdot 3^{2\alpha}+1}{8}\right) \equiv 0 \pmod{3}.
$$
We then want to prove that 
$$
pond\left(3^{2\alpha +3}n+\frac{23\cdot 3^{2\alpha + 2}+1}{8}\right) \equiv 0 \pmod{3}.
$$
Note that 
\begin{align*}
 3^{2\alpha +1}n+\frac{23\cdot 3^{2\alpha}+1}{8} 
 &= 
 3\left( 3^{2\alpha }n \right) +\frac{23\cdot 3^{2\alpha}-15+16}{8} \\
 &= 
 3\left( 3^{2\alpha }n \right) +3\left( \frac{23\cdot 3^{2\alpha -1}-5}{8}\right) + 2 \\
 &= 
 3\left( 3^{2\alpha }n  + \frac{23\cdot 3^{2\alpha -1}-5}{8}\right) + 2
\end{align*}
and it is easy to argue that 
$$3^{2\alpha }n  + \frac{23\cdot 3^{2\alpha -1}-5}{8}$$
is an integer for any $\alpha \geq 1$.  Therefore, we have the following:  
\begin{align*}
 &
 pond\left( 3^{2\alpha +1}n+\frac{23\cdot 3^{2\alpha}+1}{8} \right) \\
 &= 
 pond\left( 3\left( 3^{2\alpha }n  + \frac{23\cdot 3^{2\alpha -1}-5}{8}\right) + 2 \right) \\
 &\equiv 
 pond\left( 27\left( 3^{2\alpha }n  + \frac{23\cdot 3^{2\alpha -1}-5}{8}\right) + 17 \right) \pmod{3} \textrm{\ \ \ thanks to Theorem \ref{thm:pond_internal}} \\
 &= 
 pond\left(  3^{2\alpha +3}n  + \frac{23\cdot 3^{2\alpha +2}-27\cdot 5 +17\cdot 8}{8}  \right) \\
 &= 
 pond\left(  3^{2\alpha +3}n  + \frac{23\cdot 3^{2\alpha +2}+1}{8}  \right) \\
 &\equiv 
 0 \pmod{3}
\end{align*}
thanks to the induction hypothesis.  This completes the proof.
\end{proof}


\section{Congruences for $pend(n)$}
\label{sec:pend}
We now turn our attention to proving Theorems \ref{thm:pend_cong}, \ref{thm:pend_internal}, and \ref{thm:pend_infinite_family_congs}.
We begin by finding the generating function for $pend(n).$
\begin{theorem}
\label{thm:pend_genfn}
We have
$$
\sum_{n=0}^{\infty} pend(n)q^n =\frac{f_2f_{12}}{f_1f_4f_6}.
$$
\end{theorem}
\begin{proof}
Using the definition of the partitions counted by $pend(n)$, we 
know
\begin{align*}
\sum_{n=0}^\infty pend(n)q^n 
&= 
\frac{1}{(q;q^2)_\infty}\prod_{i=1}^\infty \left( \frac{1}{1-q^{2i}} - q^{2i} \right) \\
&= 
\frac{f_2}{f_1}\prod_{i=1}^\infty \left( \frac{1-q^{2i}+q^{4i}}{1-q^{2i}}  \right) \\
&= 
\frac{f_2}{f_1}\prod_{i=1}^\infty \left( \frac{1+q^{6i}}{(1+q^{2i})(1-q^{2i})}  \right) \\
&= 
\frac{f_2}{f_1} \frac{(-q^6;q^6)_\infty}{f_4} \\
&= 
\frac{f_2}{f_1}\cdot \frac{f_{12}}{f_4f_6} \\
&= 
\frac{f_2f_{12}}{f_1f_4f_6}.
\end{align*}
\end{proof}

We now turn our attention to proving Theorem \ref{thm:pend_cong}.  This will require that we 3--dissect the generating function for $pend(n)$ in a particular way.   

\begin{proof}(of Theorem \ref{thm:pend_cong})
Thanks to Theorem \ref{thm:pend_genfn}, we see that 
\begin{align*}
\sum_{n=0}^{\infty} pend(n)q^n 
&=
\frac{f_2f_{12}}{f_1f_4f_6} \\
&\equiv  
\frac{f_4^2}{f_1f_2^2} \pmod{3} \textrm{\ \ from Lemma \ref{lemma:binthm_divisibility}} \\
&= 
\frac{f_4^3}{f_2^3} \cdot \frac{f_2}{f_1f_4} \\
&\equiv 
\frac{f_{12}}{f_6} \cdot \frac{f_2}{f_1f_4} \pmod{3}.
\end{align*}
From Lemma \ref{f2_over_f1f4}, we then know that 
$$
\sum_{n=0}^\infty pend(3n+1)q^{3n+1} \equiv \frac{f_{12}}{f_6} \left( q\frac{f_6^2f_{18}^3}{f_3^3f_{12}^3} \right) \pmod{3}
$$ 
which means 
\begin{align}
\sum_{n=0}^\infty pend(3n+1)q^{n} 
&\equiv 
\frac{f_{4}}{f_2} \cdot \frac{f_2^2f_{6}^3}{f_1^3f_{4}^3} \pmod{3} \notag \\
&= 
\frac{f_{2}f_6^3}{f_1^3f_4^2} \notag \\
&\equiv 
\left(  f_2f_4 \right) \frac{f_6^3}{f_3f_{12}} \pmod{3} \label{3n1_genfn}.  
\end{align}
Thanks to Lemma \ref{f1f2}, we see that 
$$
\sum_{n=0}^\infty pend(3n+1)q^{n} 
\equiv \frac{f_6^3}{f_3f_{12}} \left( \frac{f_{12}f_{18}^4}{f_6f_{36}^2} +2q^2f_{18}f_{36} +q^4\frac{f_6f_{36}^4}{f_{12}f_{18}^2} \right) \pmod{3}.
$$
This now allows us to perform an additional 3--dissection to obtain 
$$
\sum_{n=0}^\infty pend(9n+1)q^{3n} 
\equiv \frac{f_6^3}{f_3f_{12}} \cdot  \frac{f_{12}f_{18}^4}{f_6f_{36}^2}   \pmod{3}
$$
which yields 
\begin{align*}
\sum_{n=0}^\infty pend(9n+1)q^{n} 
&\equiv
\frac{f_2^2f_6^4}{f_1f_{12}^2} \pmod{3} \\
&= 
\frac{f_2^2}{f_1}\cdot \frac{f_6^4}{f_{12}^2}.
\end{align*}
From Lemma \ref{f2squared_over_f1}, we can rewrite this result as 
\begin{equation}
\label{9n1_genfn}
\sum_{n=0}^\infty pend(9n+1)q^{n} 
\equiv
\left( \frac{f_6f_9^2}{f_3f_{18}} + q\frac{f_{18}^2}{f_9} \right) \frac{f_6^4}{f_{12}^2} \pmod{3}.
\end{equation}
Note that the power series representation of the right--hand side of the above congruence contains no terms of the form $q^{3n+2}.$  Thus, 
$$
\sum_{n=0}^\infty pend(9(3n+2)+1)q^{3n+2} \equiv 0 \pmod{3}
$$
which means that, for all $n\geq 0,$
$$
pend(9(3n+2)+1) = pend(27n+19) \equiv 0 \pmod{3}.
$$
\end{proof}

We next consider the proof of Theorem \ref{thm:pend_internal}.
\begin{proof}(of Theorem \ref{thm:pend_internal})
Our goal here is to prove that, for all $n\geq 0,$ 
$$
pend(27n+10) \equiv pend(3n+1) \pmod{3}.
$$  
Thanks to (\ref{9n1_genfn}), we see that 
$$
\sum_{n=0}^\infty pend(27n+10)q^{3n+1} 
\equiv q\frac{f_6^4f_{18}^2}{f_9f_{12}^2} \pmod{3}
$$
which means 
\begin{equation}
\label{27n10_genfn}
\sum_{n=0}^\infty pend(27n+10)q^{n} 
\equiv \frac{f_2^4f_{6}^2}{f_3f_{4}^2} \pmod{3}.    
\end{equation}
From
(\ref{3n1_genfn}), we know 
\begin{align*}
\sum_{n=0}^\infty pend(3n+1)q^n 
&\equiv 
\frac{f_2f_4f_6^3}{f_3f_{12}} \pmod{3} \\
&\equiv 
\frac{f_2f_4f_6^3}{f_3f_{4}^3} \pmod{3} \\
&=
\frac{f_2f_6^3}{f_3f_{4}^2} \\
&\equiv 
\frac{f_2f_2^3f_6^2}{f_3f_4^2} \pmod{3} \\
&= 
\frac{f_2^4f_{6}^2}{f_3f_{4}^2} \\
&\equiv 
\sum_{n=0}^\infty pend(27n+10)q^{n} \pmod{3}
\end{align*}
thanks to (\ref{27n10_genfn}).  
\end{proof}
We are now in a position to prove the infinite family of congruences in Theorem \ref{thm:pend_infinite_family_congs}.
\begin{proof}(of Theorem \ref{thm:pend_infinite_family_congs})
We prove this theorem by induction on $\alpha$.  Note that the base case, $\alpha=1$, which corresponds to the arithmetic progression 
$$
3^3n + \frac{17\cdot 3^{2}-1}{8} = 27n+19,
$$ 
has already been proved in Theorem \ref{thm:pend_cong}.   
Thus, we assume that, for some $\alpha\geq 1$ and all $n\geq 0$, 
$$
pend\left(3^{2\alpha +1}n+\frac{17\cdot 3^{2\alpha}-1}{8}\right) \equiv 0 \pmod{3}.
$$
We then want to prove that 
$$
pend\left(3^{2\alpha +3}n+\frac{17\cdot 3^{2\alpha + 2}-1}{8}\right) \equiv 0 \pmod{3}.
$$
Note that 
\begin{align*}
 3^{2\alpha +1}n+\frac{17\cdot 3^{2\alpha}-1}{8} 
 &= 
 3\left( 3^{2\alpha }n \right) +\frac{17\cdot 3^{2\alpha}-9+8}{8} \\
 &= 
 3\left( 3^{2\alpha }n \right) +3\left( \frac{17\cdot 3^{2\alpha -1}-3}{8}\right) + 1 \\
 &= 
 3\left( 3^{2\alpha }n  + \frac{17\cdot 3^{2\alpha -1}-3}{8}\right) + 1
\end{align*}
and it is easy to argue that 
$$3^{2\alpha }n  + \frac{17\cdot 3^{2\alpha -1}-3}{8}$$
is an integer for any $\alpha \geq 1$.  Therefore, we have the following:  
\begin{align*}
 &
 pend\left( 3^{2\alpha +1}n+\frac{17\cdot 3^{2\alpha}-1}{8} \right) \\
 &= 
 pend\left( 3\left( 3^{2\alpha }n  + \frac{17\cdot 3^{2\alpha -1}-3}{8}\right) + 1 \right) \\
 &\equiv 
 pend\left( 27\left( 3^{2\alpha }n  + \frac{17\cdot 3^{2\alpha -1}-3}{8}\right) + 10 \right) \pmod{3} \textrm{\ \ thanks to Theorem \ref{thm:pend_internal}} \\
 &= 
 pend\left(  3^{2\alpha +3}n  + \frac{17\cdot 3^{2\alpha +2}-27\cdot 3 +10\cdot 8}{8}  \right) \\
 &= 
 pend\left(  3^{2\alpha +3}n  + \frac{17\cdot 3^{2\alpha +2}-1}{8}  \right) \\
 &\equiv 
 0 \pmod{3}
\end{align*}
thanks to the induction hypothesis.  This completes the proof.
    
\end{proof}



\section{Closing Thoughts}  

While it is very satisfying to see the proofs provided above, it would be interesting to see combinatorial proofs of these divisibility properties.  We leave it to the interested reader to obtain such proofs.  

It may also be fruitful to consider further refinements of the functions $pend(n)$ and $pond(n)$.  For example, rather than requiring that even parts must be repeated, one could restrict this requirement to only those parts which are divisible by 4 (with no such requirements on the other parts).  It is certainly straightforward to find the generating functions for such refinements, which means that an analysis such as that above should be possible.  Ballantine and Welch \cite{BalWel} share comments about such partitions (and their enumerating functions) near the end of their manuscript.  The interested reader may wish to study such functions from an arithmetic perspective.  

\section*{Acknowledgements}
The author gratefully acknowledges Shane Chern for beneficial conversations during the development of this work.

\section*{Declarations}

\noindent {\bf Ethical approval} Not applicable.  
\vskip .1in 
\noindent {\bf Competing interests}  The author declares that there are no competing interests. 
\vskip .1in 
\noindent {\bf Funding}  Not applicable.  
\vskip .1in 
\noindent {\bf Availability of data and materials}  Not applicable.

\end{document}